\documentclass[11pt]{amsart}
 \usepackage{amssymb,amsmath,amscd}

\usepackage{pgf,tikz, subfigure}
\usepackage{mathrsfs}
\usepackage{epstopdf}
\usetikzlibrary{arrows}
\newtheorem{theorem}{Theorem}[section]
\newtheorem{proposition}[theorem]{Proposition}

\newtheorem*{theorem*}{Theorem}
\theoremstyle{definition}
\newtheorem{definition}[theorem]{Definition}
\newtheorem{remark}[theorem]{Remark}

\newcommand{\SSS}{{\mathbb S}}

\newcommand{\cF}{{\mathcal F}}

\begin{document}

\title[Desingularizing $b^m$-symplectic structures]{Desingularizing $b^m$-symplectic structures}

\author{Victor Guillemin}\address{ Victor Guillemin, Department of Mathematics, Massachussets Institute of Technology, Cambridge MA, US,
\it{e-mail: vwg@math.mit.edu}}
\author{Eva Miranda}\address{ Eva Miranda,
Departament de Matem\`{a}tiques, Universitat Polit\`{e}cnica de Catalunya and BGSMath Barcelona Graduate School of
Mathematics, Barcelona, Spain  \it{e-mail: eva.miranda@upc.edu}
 }
\author{Jonathan Weitsman} \address{Jonathan Weitsman, Department of Mathematics, Northeastern University, Boston MA,  US, \it{e-mail: j.weitsman@neu.edu}}
   \thanks{Eva Miranda is supported by the Catalan Institution for Research and Advanced Studies via ICREA Academia 2016 Prize, by the  Ministerio de Econom\'{\i}a y Competitividad project with reference MTM2015-69135-P and by the Generalitat de Catalunya project with reference 2014SGR634. Jonathan Weitsman  is supported in part by NSF grant DMS 12-11819. }

\date{\today}

\begin{abstract}
A $2n$-dimensional Poisson manifold $(M ,\Pi)$ is said to be $b^m$-symplectic if it is symplectic on the complement of a hypersurface $Z$ and has a simple
Darboux canonical form  at points of $Z$ which we will describe below. In this paper we will discuss a desingularization procedure which,
for $m$ even, converts $\Pi$ into a family of symplectic forms $\omega_{\epsilon}$ having the property that $\omega_{\epsilon}$ is equal to
the $b^m$-symplectic form dual to $\Pi$ outside an $\epsilon$-neighborhood of $Z$ and, in addition, converges to this form as $\epsilon$ tends
to zero in a sense that will be made precise in the theorem below. We will then use this construction to show that a number of somewhat
mysterious properties of $b^m$-manifolds can be more clearly understood by viewing them as limits of analogous properties of the
$\omega_{\epsilon}$'s. We will also prove versions of these results for $m$ odd;  however, in the odd case the family $\omega_{\epsilon}$ has
to be replaced by a family of  \lq\lq folded\rq\rq   symplectic forms.

\end{abstract}

\maketitle

\section{Introduction}

 A $b$-symplectic manifold is an oriented Poisson manifold $(M ,\Pi)$ which has the property that the map $\Pi^n: M \longrightarrow
 \Lambda^{2n}(TM)$
intersects the zero section of $\Lambda^{2n}(TM)$ transversally in a codimension one submanifold $Z\subset M$. For such a Poisson manifold the
dual to the bivector field $\Pi$ is a generalized De Rham form of $b$-type and defines a \lq\lq $b$-symplectic\rq\rq  structure on $M$.
{The $b$ in \lq\lq $b$-symplectic\rq\rq  comes from \emph{boundary} as these structures were initially considered by Nest and Tsygan \cite{nestandtsygan} in connection to deformation quantization of manifolds with boundary. The theory of calculus on manifolds with boundary, or $b$-calculus, as developed by Melrose \cite{melrose}\footnote{for his proof of the Atiyah-Patodi index theorem.} is used in this study. Notwithstanding, in this article the considered manifolds  are manifolds without boundary and the critical set $Z$ is an hypersurface of the manifold.}

 {$b$-Symplectic structures and their applications} have been the topic of a number of recent articles (see \cite{guimipi},\cite{guimipi2}, \cite{Gualtierili},
\cite{marcutosorno1}, \cite{frejlichmartinezmiranda}, \cite{cavalcanti}, \cite{marcutosorno2}, \cite{gmps}, \cite{gualtierietal} \cite{kms},
\cite{gmps2}) and generalizations of these structures in which one no longer requires the transversality assumption above have also been
considered.  In this paper we will be concerned with one such generalization, due to Geoffrey Scott \cite{scott} in which the transversality
assumption is replaced by the assumption that away from $Z,$ $M$ is symplectic while at $Z$ the Poisson structure has a {simple local expression}.

 { For $b^m$-symplectic structures  a Darboux normal form can be obtained as in the classical Darboux theorem for symplectic structures but replacing one of the smooth coordinates by a $b^m$-function. So locally we can write }{$$\displaystyle\frac{dx_1}{x_1^m}\wedge dy_1+\displaystyle\sum_{2}^{n} dx_i\wedge dy_i.$$}  These structures  known as $b^m$ structures  {and their \lq\lq dual" folded-symplectic structures  show up modeling problems in Celestial Mechanics as the three-body problem or their restricted versions whenever ad-hoc coordinates transformations such as McGehee coordinates or Kustaanheimo-Stiefel regularizations are considered \cite{kms,dkm}.}

 {Our goal in this paper is to  prove that such a duality exhibited in examples holds globally on the manifold. In other words, in this paper we describe a} \lq\lq
desingularization\rq\rq  of these structures: Where $m$ is even, we will construct a family of symplectic forms
on $M$, depending on a parameter $\epsilon,$ and having the property that as $\epsilon$ tends to zero these forms tend in the limit to the $b^m$
form that is the dual object to the $b^m$ Poisson bivector field $\Pi$.  Where $m$ is odd, we prove an analogous result, but with the family of
symplectic forms replaced by a family of folded symplectic forms.  More explicitly we prove (Theorems \ref{thm:desingularizingsymp} and
\ref{thm:desingularizingfolded}):
\begin{theorem*}
{ Let $\omega$ be a} $b^m$-symplectic structure  on a compact manifold $M$  and let $Z$ be its critical hypersurface.
\begin{itemize}
\item If $m$ is even, there exists  a family of symplectic forms ${\omega_{\epsilon}}$ which coincide with  the $b^{m}$-symplectic form
    $\omega$ outside an $\epsilon$-neighborhood of $Z$ and for which  the family of bivector fields $(\omega_{\epsilon})^{-1}$ converges in
    the $C^{2k-1}$-topology to the Poisson structure $\omega^{-1}$ as $\epsilon\to 0$ .
\item If $m$ is odd, there exists  a family of folded symplectic forms ${\omega_{\epsilon}}$ which coincide with  the $b^{m}$-symplectic form
    $\omega$ outside an $\epsilon$-neighborhood of $Z$.

\end{itemize}

\end{theorem*}

{ One of the main concerns in Symplectic Geometry is to determine the topological constraints for a manifold to admit a symplectic structure. In this spirit, the theorems that we prove in this paper imply that any manifold admitting a $b^{2k}$-symplectic structure also admits a honest symplectic structure and therefore all the topological constraints for a manifold to be symplectic also apply for a manifold to admit a $b^{2k}$-symplectic manifold.}

 {The same reasoning applies for $b^{2k+1}$-symplectic structures and folded symplectic structures though not every folded symplectic manifold admits a $b^{2k+1}$-symplectic structure since the subclass of desingularized folded symplectic structures is strict. In particular Cannas da Silva proved that every $4$-manifold admits a folded symplectic manifold but as proved in \cite{guimipi, marcutosorno1} not every $4$-manifold admits a $b$-symplectic structure as some topological constraints apply\footnote{The same applies for $b^{2k+1}$-symplectic manifolds since the arguments in \cite{marcutosorno1} can be adapted to the  $b^{2k+1}$-case}. In particular $S^4$ does not admit a $b$-symplectic structure but it admits a folded symplectic structure.}

 Our goal in introducing these families is that in attempting to define $b^m$-{analogs} of a number of basic invariants of symplectic and folded
 symplectic manifolds such as symplectic volume and, for Hamiltonian $G$ manifolds, moment polytopes and Duistermaat-Heckman measures, one
 encounters a number of frustrating \lq\lq infinities\rq\rq that are hard to interpret or eliminate {as it was already seen for instance in \cite{gmps, gmps2}}. However, we believe (and will verify below
 in a couple of important cases) that desingularization is an effective tool for getting around this problem.

{\textbf{Acknowledgements:} We are indebted to the three referees of this paper for suggesting improvements. We are also grateful to Arnau Planas for helping us with the pictures in this paper.}
\section{Preliminaries}

Let $M$ be a compact manifold and let $Z\subset M$ a hypersurface in $M$.
In \cite{guimipi} a  $b$-symplectic form was defined {as} a $2$-form  in the complex of $b$-forms. In order to define this complex we first
considered the $b$-tangent bundle $^b T(M)$ (whose sections are defined as vector fields tangent to the critical hypersurface $Z$)  and its dual
$^b T^*(M)$. The complex of $b$-forms was introduced \`{a} la De Rham as sections of the bundles $\Lambda^{k} ({^b} T^*(M))$.

In \cite{scott} a similar description was obtained for $b^m$-symplectic forms. Let $M$ be a compact oriented manifold and let $Z \subset M$ be a
hypersurface in $M,$ along with a choice of function $x \in C^\infty(M)$ such that $0$ is a regular value of $x$ and $x^{-1}(0) = Z.$
Given such a triple $(M,Z,x),$  the fibers of the $b^m$-(co)tangent bundle are given by
\begin{align*} {^{b^m}}T_pM &\cong \left\{ \begin{array}{c l}T_pZ + < x^m\frac{\partial}{\partial x} >  & \textrm{if} \ p \in Z  \\T_pM &
\textrm{if} \ p \notin Z
\end{array}\right.\\ {^{b^m}}T_p^*M &\cong
\left\{ \begin{array}{c l}T_p^*Z + < \frac{dx}{x^m} > & \textrm{if} \ p \in Z  \\T_p^*M & \textrm{if} \ p \notin Z \end{array}\right.
\end{align*}

 \noindent As in the case of $b$-manifolds, these fibres combine to form a bundle; a $b^m$-manifold is a triple $(M,Z,x),$ along with these
 bundles.\footnote{By abuse of notation, we denote a $b^m$-manifold by $(M,Z),$ suppressing the function $x.$  Note that Scott \cite{scott}'s
 definition of a $b^m$-manifold differs from ours by allowing {\em local} defining functions for $Z.$}

We then define

 \begin{definition} A {\bf symplectic $b^m$-manifold}  is a $b^m$-manifold $(M, Z)$ with a closed $b^m$-two form $ \omega $  which has maximal
 rank  at every $p \in M$.
\end{definition}

To describe the properties of such forms we will need the following definitions and propositions (see  {Section 4} in \cite{scott}).

\begin{definition} A {\bf Laurent Series} of a closed $b^m$-form $\omega$ is a decomposition of $\omega$ in a tubular neighborhood $U$ of $Z$
of the form
\begin{equation}\label{eqn:laurent}
\displaystyle\omega = \frac{dx}{x^m} \wedge \left(\sum_{i = 0}^{m-1}\pi^*(\hat{\alpha_{i}})x^i\right) + \beta
\end{equation}

\noindent where $\pi: U \to Z$ is the projection, where each $\hat{\alpha_{i}}$ is
a closed smooth De Rham form on $Z,$ and $\beta$ is a De Rham form on {$U$}. \end{definition}

Proposition 3.3 and Remark 3.4 in \cite{scott} yield,
\begin{proposition}[Scott]\label{prop:laurent}
In a { tubular} neighborhood of $Z$, every   closed $b^m$-form  $\omega$ can be written in a  Laurent form of type (\ref{eqn:laurent}) { and the restriction of $\sum_{i = 0}^{m-1}\pi^*(\hat{\alpha_{i}})x^i$ and $\beta$ to $Z$ are well-defined closed $1$ and $2$-forms respectively. }
\end{proposition}

\subsection{ Symplectic foliations and normal forms for $b^m$-symplectic manifolds}

We begin by studying the symplectic foliation of the Poisson structure induced by  a $b^m$-symplectic form on the critical hypersurface $Z$.

{\begin{proposition}\label{prop:foliation} Given a  $b^m$-symplectic manifold with $b^m$ symplectic form $\omega,$ the closed one-form $\hat{\alpha_0}$
in the Laurent decomposition
\[ \omega = \frac{dx}{x^m}\wedge(\sum_{i = 0}^{m-1}  \pi^*(\hat{\alpha_{i}})x^i) + \beta \]
defines the codimension-one symplectic foliation $\mathcal F$ of the regular Poisson structure induced  on
the critical hypersurface $Z$. The symplectic form on each leaf is given by the restriction of  the $2$-form $\beta$.  In addition one can find  a Poisson vector field $v$ on $Z$ transverse to this foliation.
\end{proposition}}

See { section 4.1} in \cite{guimipi2} for the proof of this in the $m=1$ case.
{For $m>1$ the proof is essentially the same and it only requires the identification of the symplectic foliation associated to the dual Poisson structure.}
\begin{remark}
 { The vector field $v$ is the  unique vector field  satisfying the equations
$$\begin{cases}{\begin{array}{ccl}\iota_v\hat{\alpha_0}&=&1 \\ \iota_v\beta&=&0\end{array}}\end{cases}$$  \noindent where $\hat{\alpha_0}$ is the defining
one-form for the foliation $\mathcal F$ and $\beta$ a closed $2$-form on $Z$ that restricts to the symplectic form on every leaf of $\mathcal
F$.}
\end{remark}

{\begin{remark}
As in \cite{guimipi} the pair of  forms $(\hat{\alpha_0}, \beta)$ defines a cosymplectic structure on $Z$.
\end{remark}}

%
%
%
%

 By Theorem 13 in \cite{guimipi} one gets,

\begin{theorem}\label{globalcompact} If $\cF$ contains a compact leaf $L$, then every leaf of $\cF$ is diffeomorphic to $L$,  and  $Z$ is
the total space of a fibration $f:M\to\SSS^1$ with fiber $L$, and $\cF$ is the fiber foliation $\{f^{-1}(\theta)|\theta\in\SSS^1\}$.
\end{theorem}

As it was proved in Corollary 4 in \cite{guimipi}  if  the manifold $Z$ is compact then it is the mapping torus of the map $\phi:L\to L$  given
by the holonomy map of the fibration over $\SSS^1$, $\frac{L\times \left[0,1\right]}{ (x,0)\sim (\phi(x),1)}$ with $\phi$ is  the first return map of $\exp tv$ and $v$ the vector field above.

\subsection{ $b^m$-versions of the Moser and Darboux theorems} For $m=1$ the statement and proof of these results can be found in
\cite{guimipi2}.  The proof in \cite{guimipi2} is based on  the Moser path method for $b$-symplectic structures; however, the Moser path method
also works for $b^m$-symplectic structures (see Theorems 5.2 and 5.3 in  \cite{scott}), so the results apply to $ b^m-$symplectic manifolds as well. The first of these
theorems asserts

\begin{theorem}\label{theorem:moserbn} If $\omega_0, \omega_1$ are symplectic $b^m$-forms on $(M, Z)$ with $Z$ compact and
$\omega_0\vert_{{Z}} = \omega_1\vert_{Z}$, then there are neighborhoods $U_0, U_1$ of $Z$ and a $b^m$-symplectomorphism $\varphi: (U_0, Z,
\omega_0) \rightarrow (U_1,Z,  \omega_1)$ such that $\varphi\vert_{Z} = Id.$ \end{theorem}
(For the proof see \cite{guimipi2}, Theorem 6.5).
Consider  now the decomposition of a $b^m$ form
\begin{equation}\label{eq:splittingbnsymplecticform}\omega=\alpha\wedge\frac{dx}{x^m}+\beta, \text{ with } \alpha\in\Omega^1(M) \text{ and }
\beta\in\Omega^2(M).
\end{equation}

\begin{remark} {Observe that this theorem being of semilocal flavor holds in a neighborhood of a compact submanifold, independently of the compactness of the ambient manifold. Also the Moser theorem in \cite{scott} is global and the semilocal applications of it (cf. Theorem 5.2 in \cite{scott}) contain unnecessary hypotheses on the cohomology classes of the $b^m$-symplectic form (compare with Theorem 34 in \cite{guimipi}). }

 \end{remark}
 To prove the $b^m$-version of the Darboux theorem we will need,

\begin{proposition}\label{prop:thetamu}(See Proposition 10 in \cite{guimipi})
Let $\tilde{\alpha}=i^*\alpha$ and $\tilde{\beta}=i^*\beta$, where $i:Z\hookrightarrow M$ denotes the inclusion.  Then
the forms $\tilde{\alpha}$ and $\tilde{\beta}$ are closed. Furthermore,
 \begin{enumerate}
\item[a)] The form $\tilde{\alpha}$ is nowhere vanishing and intrinsically defined in the sense that it does not depend on the splitting
    (\ref{eq:splittingbnsymplecticform}). In particular, the codimension-one foliation of $Z$ defined by $\tilde{\alpha}$ is intrinsically
    defined.
\item[b)] For each leaf $L\stackrel{i_L}\hookrightarrow Z$ of this foliation, the form $i^*_L\tilde{\beta}$ is intrinsically defined, and is a
    symplectic form on $L$.
\item[c)]In (\ref{eq:splittingbnsymplecticform}) we can assume without loss of generality that:
\begin{itemize}
\item The forms $\alpha$ and $\beta$ are closed.
\item The form $\alpha\wedge\beta^{n-1}\wedge d{x}$ is nowhere vanishing.
\item And, in particular, the form $i^*(\alpha\wedge\beta^{n-1})$ is nowhere vanishing.
\end{itemize}
\end{enumerate}
\end{proposition}

{ The same proof can be adapted for $m>1$ since $\tilde{\alpha}=\tilde{\alpha_0}$.}
We will now show

\begin{theorem}[$b^m$-Darboux theorem]\label{theorem:Darbouxbn}
Let $\omega$ be a $b^m$-symplectic form on $(M,Z)$ and $p\in Z$. Then we can find a coordinate chart $(U,x_1,y_1,\ldots,x_n,y_n)$ centered at
$p$ such that on $U$ the hypersurface $Z$ is locally defined by $x_1=0$ and
$$\omega=\frac{d x_1}{x_1^m}\wedge d y_1+\sum_{i=2}^n d x_i\wedge d y_i.$$
\end{theorem}

\begin{proof} Write $\omega=\alpha\wedge \frac{d x_1}{x_1^m}+\beta$, and $\tilde{\alpha}=i^*\alpha$ and $\tilde{\beta}=i^*\beta$, with
$i:Z\hookrightarrow M$ the inclusion. From Proposition \ref{prop:thetamu}, for all $p\in Z,$ we have $\tilde{\alpha}_p$ non-vanishing. Thus
$\tilde{\alpha}_p\wedge\tilde{\beta}_p\neq0$ and $\tilde{\beta}_p\in\Lambda^2( T_p^*Z)$ has rank $n-1$.
Thus we can assume
$$\omega|_Z=(\frac{ d x_1}{x_1^m}\wedge d y_1+\sum_{i=2}^n d x_i\wedge d y_i)|_Z.$$ \noindent and the assertion above follows from
 Theorem \ref{theorem:moserbn}.
\end{proof}

\section{Desingularizing $b^{2k}$-symplectic structures}

Consider a manifold $M$ equipped with a $b^{2k}$-symplectic structure given by a $b^{2k}$-symplectic form $\omega.$ In view of the Laurent
decomposition given in Proposition \ref{prop:laurent}, we have in a tubular neighborhood $U$ of $Z$

\begin{equation}\label{00} \omega = \frac{dx}{x^{2k}}\wedge(\sum_{i = 0}^{2k-1}  \alpha_{i}x^i) + \beta \end{equation}

\noindent where $\alpha_i=\pi^*(\widehat{\alpha_i})$ with $\widehat{\alpha_i}$ closed one-forms on $Z$ and $\pi :U \to Z$ denoting the
projection.

Let $f\in \mathcal{C}^{\infty}(\mathbb R)$ be an odd smooth function satisfying $f'(x)>0$ for all $x \in [-1,1]$ as shown below,

\begin{figure}[!htb]
\centering
\includegraphics[scale=.6]{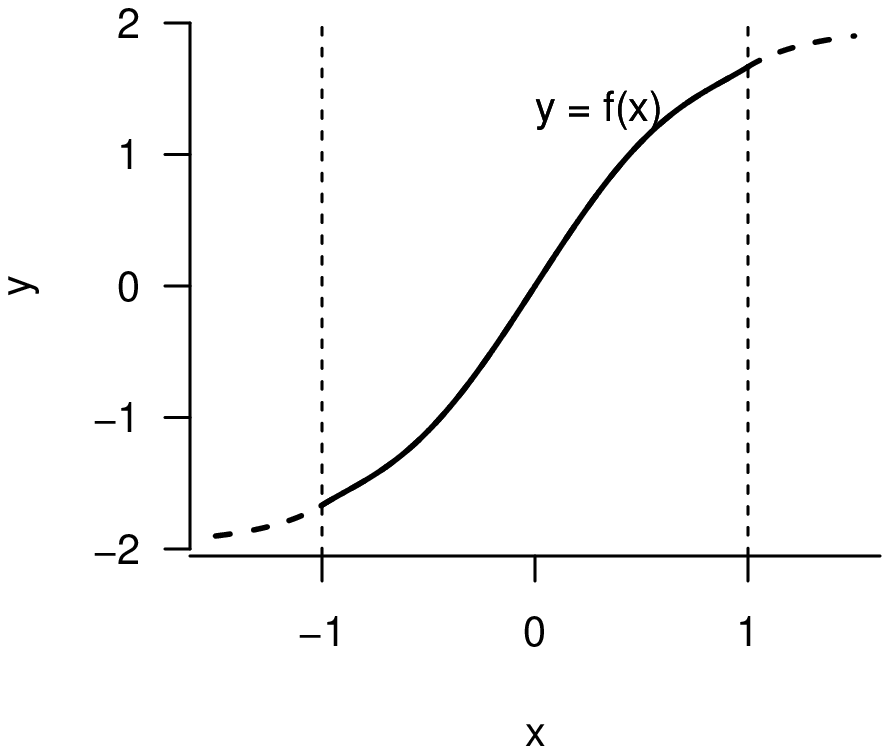}
\label{fig:digraph}
\end{figure}
\noindent and satisfying

\[f(x)=\begin{cases} \frac{-1}{(2k-1) x^{2k-1}}-2 &\textrm{for} \quad x<-1  \\ \frac{-1}{(2k-1) x^{2k-1}}+2 &\textrm{for} \quad x>1
\end{cases}\]

\noindent outside the interval $[-1,1]$.

Now we scale the function $f$ to construct a new function \begin{equation}\label{definingequation}f_\epsilon(x):=  \frac{1}{\epsilon^{2k-1}}
f\!\left(\frac{x}{\epsilon}\right).\end{equation} Thus outside the interval $[-\epsilon,\epsilon]$ ,

\[f_\epsilon(x)=\begin{cases} \frac{-1}{(2k-1) x^{2k-1}}-\frac{2}{\epsilon^{2k-1}} &\textrm{for} \quad x<-\epsilon  \\ \frac{-1}{
(2k-1)x^{2k-1}}+\frac{2}{\epsilon^{2k-1}}&\textrm{for} \quad x>\epsilon
\end{cases}\]
 We replace $\frac{dx}{x^{2k}}$ by $df_\epsilon$  in the expansion (\ref{00}) an $\epsilon$-neighborhood and obtain a differential form

\[ {\omega_\epsilon} = df_\epsilon\wedge(\sum_{i = 0}^{2k-1} \alpha_{i}x^i) + \beta \]

\noindent Since $\omega_\epsilon$ agrees with $\omega$ outside an $\epsilon$ neighborhood of $Z,$ it extends to a differential form on all of
$M.$
We denote this extension (by abuse of notation) by ${\omega_\epsilon},$ the \textbf{$f_\epsilon$-desingularization}\footnote{{which we can also refer to as {\bf
deblogging} since it eliminates the $b/log$-singularity} } of the  $b^{2k}$-symplectic structure $\omega.$

\begin{theorem}\label{thm:desingularizingsymp}
The  \textbf{$f_\epsilon$-desingularization} ${\omega_{\epsilon}}$ is symplectic. The family ${\omega_{\epsilon}}$  coincides with  the
$b^{2k}$-symplectic form $\omega$ outside an $\epsilon$-neighborhood.  The family of bivector fields $\omega_{\epsilon}^{-1}$ converges to the
Poisson structure $\omega^{-1}$  in the $C^{2k-1}$-topology as $\epsilon\to 0$ .
\end{theorem}

As a consequence of this theorem we obtain,

\begin{theorem}
A manifold admitting a  $b^{2k}$-symplectic structure also admits a symplectic structure.
\end{theorem}

In particular the topological constraints that apply for symplectic structures also apply for $b^{2k}$-symplectic structures.

This point of view in the study of $b^m$-symplectic forms yields several consequences. In this paper we concentrate on a couple of them
concerning volume forms and Hamiltonian actions.

We now prove Theorem \ref{thm:desingularizingsymp}.

\begin{proof} Clearly for all $\epsilon$,  the form ${\omega_\epsilon} = df_\epsilon\wedge(\sum_{i = 0}^{2k-1}  \alpha_{i}x^i) + \beta$ is
closed since all the one forms $\alpha_i$  and the $2$-form $\beta$ are closed.

Let us check that it is symplectic.  Outside $U,$ $\omega_\epsilon$ coincides with $\omega.$  In $U$ but away from $Z$, $$\omega_\epsilon^n =
\frac{df_\epsilon}{dx}{ x^{2k}} \omega^n$$ which is nowhere vanishing.  To check that $\omega_\epsilon$ is symplectic at $Z,$ observe that
$${\omega_\epsilon}= df_{\epsilon}\wedge (\sum_{i=0}^{2k-1} x^i\alpha_i)+\beta=\epsilon^{-2k} f^{\prime}\!\left(\frac{x}{\epsilon}\right) dx \wedge
(\sum_{i=0}^{2k-1} x^i \alpha_i)+\beta$$

\noindent which on the interval
 $\vert x\vert <\epsilon$ is equal to
 {$\epsilon^{-2k}( f^{\prime}\!\left(\frac{x}{\epsilon}\right) dx \wedge \alpha_0  +\mathcal{O}(\epsilon))+ \beta$}
\noindent and hence  {$${\omega_\epsilon}^n=\epsilon^{-2k}(f^{\prime}\!\left(\frac{x}{\epsilon}\right) dx \wedge
\alpha_0\wedge\beta^{n-1}+\mathcal{O}(\epsilon))$$  which is non-vanishing for $\epsilon$ sufficiently small.
  because $dx\wedge\alpha_0\wedge \beta^{n-1}\neq 0$  by Proposition \ref{prop:thetamu} and $f^{\prime}\neq0$ by definition of $f$. So we conclude that $\omega_{\epsilon}$ is symplectic.}
%

Let us now prove that the family of bivector fields $\omega_{\epsilon}^{-1}$ converges to $\omega^{-1}$  when $\epsilon\to 0$ in the
$C^{2k-1}$-topology.

Consider the form $\omega$ and the family $\omega_{\epsilon}$. Then  in $b^{2k}$-Darboux coordinates (Theorem \ref{theorem:Darbouxbn}),
$$\omega_{\epsilon}=\epsilon^{-2k}f'\!\left(\frac{x}{\epsilon}\right) dx\wedge dy+ dx_2\wedge dy_2+\dots +dx_n\wedge dy_n$$
and $$\omega=\frac{1}{x^{2k}}dx\wedge dy+ dx_2\wedge dy_2+\dots +dx_n\wedge dy_n$$

We wish to verify that the family $\omega_{\epsilon}^{-1}$ of bivector fields  given by
\begin{equation}\label{eqn:deblogbivector} \omega_{\epsilon}^{-1}=\epsilon^{2k}g\!\left(\frac{x}{\epsilon}\right) \frac{\partial}{\partial x}\wedge
\frac{\partial}{\partial y}+  \frac{\partial}{\partial x_2}\wedge \frac{\partial}{\partial y_2}+\dots + \frac{\partial}{\partial x_n}\wedge
\frac{\partial}{\partial y_n}\end{equation}
\noindent where $g(x)=\frac{1}{f'(x)}$,
converges to \begin{equation}\omega^{-1}=\label{eqn:bivector}x^{2k}\frac{\partial}{\partial x}\wedge \frac{\partial}{\partial y}+
\frac{\partial}{\partial x_2}\wedge \frac{\partial}{\partial y_2}+\dots + \frac{\partial}{\partial x_n}\wedge \frac{\partial}{\partial
y_n}\end{equation}
\noindent as $\epsilon$ tends to zero.

{The bivector field $\omega_{\epsilon}^{-1}$ converges to the bivector field $\omega^{-1}$ in the $C^{2k-1}$ topology if $\epsilon^{2k} g\left(\frac{x}{\epsilon}\right)$  converges to $x^{2k}$ in the $C^{2k-1}$ topology. Since  $\epsilon^{2k} g\left(\frac{x}{\epsilon}\right)$  is identically equal to $x^{2k}$ for $\vert x\vert >\epsilon$, it suffices to check that on the interval $[-\epsilon, \epsilon]$, the $(2k-1)^{th}$ derivative of  $\epsilon^{2k} g\left(\frac{x}{\epsilon}\right)$  converges to $(2k)! x$ in the uniform norm. }

{ On $[-\epsilon,\epsilon]$ the function $(2k)! x$ is bounded by a constant multiple of $\epsilon$. On the other hand,
$$\displaystyle \left( \epsilon^{2k}g\left(\frac{x}{\epsilon}\right)\right)^{(2k-1)}= \epsilon g^{(2k-1)}\left(\frac{x}{\epsilon}\right)$$
\noindent with $g^{(2k-1)}(t)$ a smooth function on $-1<t<1$. Since both $(2k)! x$ and $\left( \epsilon^{2k}g\left(\frac{x}{\epsilon}\right)\right)^{2k-1}$
are bounded by constant multiples of $\epsilon$ on $[-\epsilon, \epsilon]$, so is their difference, which gives us the desired convergence.}

%
%

\end{proof}

\section{ Desingularization and volume formulae}

\subsection{ Volume formulae for $b^{2k}$-symplectic manifolds} We recall from section 5.1 in \cite{scott} the following construction which
relates the volume with the Laurent decomposition of a $b^m$-symplectic structure.

 On a tubular neighborhood {$\mathcal{U}=Z\times (-1, 1)$}, $\omega=\frac{dx}{x^{2k}}\wedge (\sum_{i=0}^{2k-1} x^i\alpha_i) +\beta$.  Hence for
 $\mathcal{U}_{\epsilon}=Z\times (-\epsilon, \epsilon)$, the symplectic volume of $M\setminus \mathcal{U}_\epsilon$ is, {up to a bounded
 term pertaining to the volume of $M\setminus \mathcal{U}$}, given by
\begin{equation}\label{eqn:volume1}
\sum_{i=0}^{2k-1}\int_{\mathcal{U}-\mathcal{U}_{\epsilon}}\frac{dx}{x^{2k-i}}\wedge \alpha_i \wedge\beta^{n-1}
\end{equation}

Furthermore,\begin{equation}\label{eqn:eq1}
\beta= dx\wedge \gamma+ \sum_{j=0}^{2k-1}x^j\pi^*(\beta_j) +\mathcal{O}(x^{2k})
\end{equation}
where $\beta_j$ are $2$-forms on $Z$. { For the sake of simplicity, from now own we will denote $\pi^*(\beta_j)$ as $\beta_j$}. Plugging equation (\ref{eqn:eq1}) into equation (\ref{eqn:volume1})  we get,
\begin{equation}\label{eqn:volume2}
\sum_{i=0}^{2k-1}\int_{I_\epsilon} \left(\frac{dx}{x^{2k-i}}\int_Z \alpha_i \wedge\left(\sum_{j=0}^{2k-1}x^j\beta_j\right)^{n-1}\right) +\mathcal{O}(1)
\end{equation}
\noindent where  $ I_\epsilon=(-1, -\epsilon)\cup(\epsilon, 1)$.  Thus,
\begin{equation}\label{eqn:volume2}
\int_{M\setminus \mathcal{U}_\epsilon} \omega^n=\sum_{i=1}^k c_i\epsilon^{-(2i-1)}+\mathcal{O}(1)
\end{equation}
\noindent where the $c_i$ are linear combinations of the integrals
$$\int_Z\alpha_{j_1}\wedge\beta_{j_2}\wedge\dots\wedge \beta_{j_n}$$

\subsection{ The desingularized version} Let us compute the symplectic volume of $M$ with respect to the symplectic form
\begin{equation}\label{eqn:eq2}
{\omega_{\epsilon}}=df_{\epsilon}\wedge (\sum_{i=0}^{2k-1} x^i\alpha_i)+\beta
\end{equation}

Outside the tubular neighborhood, $\mathcal{U}_\epsilon$, ${\omega_{\epsilon}}$ coincides with $\omega,$ so we get, for the integral of

 ${\omega_{\epsilon}}^n$ over the complement of this tube neighborhood, the result described above. What about the integral on the tube
 neighborhood?

Recall that

$$f_{\epsilon}(x)=\epsilon^{-(2k-1)}f\!\left(\frac{x}{\epsilon}\right)$$

\noindent where $f$ is the function defined in (\ref{definingequation}) and thus the integral of ${\omega_{\epsilon}}^n $ over
$\mathcal{U}_{\epsilon}$ is given by

\begin{equation}\label{eqn:eq5}
\int_{\mathcal{U}_{\epsilon}} df_{\epsilon}\wedge (\sum_{i=0}^{2k-1} x^i\alpha_i)\wedge \beta^{n-1},
\end{equation}
\noindent which by equation (\ref{eqn:eq1}) can be rewritten as
\begin{equation}\label{eqn:eq5}
\sum_{i=1}^{2k-1} b_i \int_{-\epsilon}^{\epsilon} \frac{ df_{\epsilon}}{dx} x^i dx
\end{equation}
\noindent plus a bounded error term where the coefficients $b_i$ like the $c_i$ are linear combinations of the integrals
$$\int_Z\alpha_{j_1}\wedge\beta_{j_2}\wedge\dots\wedge \beta_{j_n}.$$

To evaluate the integrals
$$ \int_{-\epsilon}^{\epsilon} \frac{ df_{\epsilon}}{dx} x^i dx $$
\noindent we make the change of coordinates $x=\epsilon y$ which converts the integral above into

\begin{equation}\label{eqn:eq6}
\epsilon^{-(2k-1)+i}\int_{-1}^{1} \frac{ df}{dy}(y) y^i d{y}.
\end{equation}

Therefore since $f(-y)=-f(y)$ this {quantity} is zero for $i$ odd and equal to a positive constant multiple of $\epsilon^{-(2k-1)+i} $ for $i$
even. Thus,
\begin{equation}\label{eqn:eq7}
\int_{\mathcal{U}_{\epsilon}} \omega_{\epsilon}^n=\sum_{i=1}^k a_i \epsilon^{-(2i-1)}
\end{equation}
\noindent where the $a_i$'s like the $b_i$'s and $c_i$'s  are linear combinations of integrals of type
$$\int_Z\alpha_{j_1}\wedge\beta_{j_2}\wedge\dots\wedge \beta_{j_n}.$$

Finally combining equations (\ref{eqn:eq7}) and equation (\ref{eqn:volume2})  this proves\begin{theorem}\label{prop:expand}
The volume determined by the desingularized symplectic form ${\omega_{\epsilon}}$ is given by  $$\int_M {\omega_{\epsilon}}^n=
\sum_{i=1}^k(a_i+c_i) \epsilon^{-(2{i}-1)}+\mathcal{O}(1)$$
\noindent where the coefficients $a_i$'s and $c_i$'s are linear combinations of integrals of type
$\int_Z\alpha_{j_1}\wedge\beta_{j_2}\wedge\dots\wedge \beta_{j_n}.$

\end{theorem}

\subsection{Leading terms} { The results in this section tell us how the symplectic volume grows asymptotically when $\epsilon\to 0$.}  The leading term in the asymptotic expansion given by formula (\ref{eqn:volume2}) is
\begin{equation}\label{eqn:eq10}
\frac{2}{2k-1} \epsilon^{-(2k-1)}\int_Z\alpha_0\wedge\beta_0^{n-1}
\end{equation}
\noindent and the leading term in the asymptotic expansion inside $\mathcal{U}_\epsilon$ is
\begin{equation}\label{eqn:eq11}
4\epsilon^{-(2k-1)}\int_Z\alpha_0\wedge\beta_0^{n-1}
\end{equation}
\noindent so adding equations (\ref{eqn:eq10}) and (\ref{eqn:eq11})
we obtain the following asymptotic result for the symplectic volume of $M$ with respect to
${\omega_\epsilon}:$

\begin{theorem}[Asymptotics for the symplectic volume] { For $(M,\omega)$ a $2n$-dimensional $b^{2k}$-symplectic manifold, the asymptotics of the symplectic volume of the family $(M,\omega_{\epsilon})$ as $\epsilon\to 0$ are:}
$$\int {\omega_\epsilon}^n\sim 2(2+ \frac{1}{2k-1})\epsilon^{-(2k-1)}\int_Z\alpha_0\wedge\beta_0^{n-1}$$
\end{theorem}

\section{Desingularizing  $b^{2k+1}$-symplectic structures}

Let $M$ be a  $b^{2k+1}$-symplectic manifold. In view of the Laurent decomposition given in Proposition \ref{prop:laurent} in an
$\epsilon$-neighborhood of $Z$ the $b$-symplectic form has the decomposition, in local coordinates

\[ \omega = \frac{dx}{x^{2k+1}}\wedge(\sum_{i = 0}^{2k}  \pi^*(\alpha_{i})x^i) + \beta \]

Let $f\in \mathcal{C}^{\infty}(\mathbb R)$ satisfy

\begin{itemize}
\item $f>0$.
\item $f(x)=f(-x)$.
\item  $f'(x)>0$ if $x<0$.
\item $f(x)=-x^2+2$ if $x\in [-1,1]$.
\item $f(x)=\log(\vert x \vert)$ if $k=0$, $x\in \mathbb R\setminus[-2,2]$.
\item $f(x)=\frac{-1}{(2k+2)x^{2k+2}}$ if $k>0$, $x\in \mathbb R\setminus[-2,2]$.

\end{itemize}
{
as depicted here:}
\begin{figure}[!htb]
\centering
\includegraphics[scale=.55]{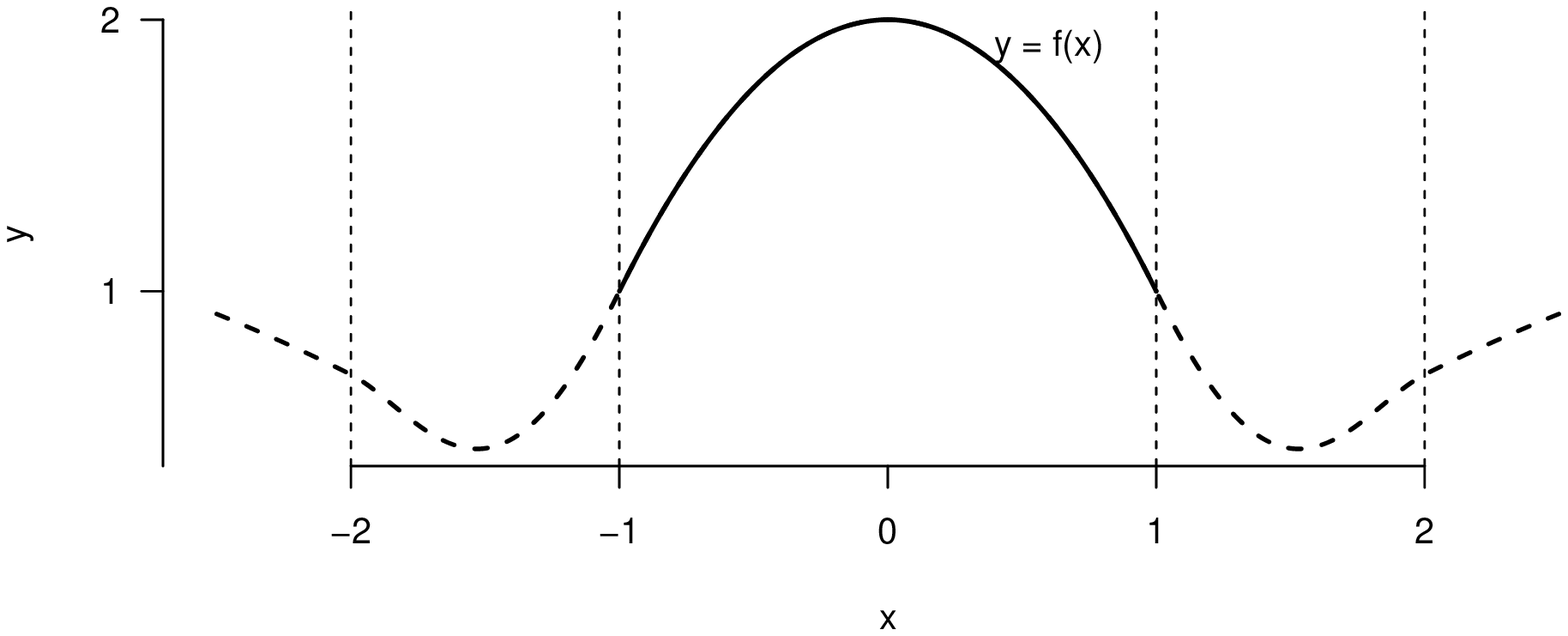}

\label{fig:digraph}
\end{figure}

Now define
\begin{equation}f_\epsilon(x):=  \frac{1}{\epsilon^{2k}} f\!\left(\frac{x}{\epsilon}\right)\end{equation}

\noindent and, as in the even case, let

\begin{equation}\label{eqn:even}
 {\omega_\epsilon} = df_\epsilon\wedge(\sum_{i = 0}^{2k}  \pi^*(\alpha_{i})x^i) + \beta
\end{equation}

We can prove the following

\begin{theorem}\label{thm:desingularizingfolded}The  $2$-form $\omega_{\epsilon}$ is a folded symplectic form which coincides with $\omega$
outside an $\epsilon$-neighborhood of $Z$.

\end{theorem}

\begin{proof}
By definition of the function $f_\epsilon$, $\omega_{\epsilon}$ coincides with $\omega$ outside an $\epsilon$-neighborhood  of the critical
hypersurface $Z$. As in the proof of Theorem  \ref{thm:desingularizingsymp}, it is easy to see that $\omega_\epsilon$ is symplectic away from
$Z.$ In order to check that $\omega_{\epsilon}$ is a folded symplectic structure, we need to check that ${\omega_{\epsilon}}^n$ is transverse to
the zero section of the bundle $\Lambda^n(T^*M)$ {and also that $(i^*{\omega_{\epsilon}})^{n-1}$ is non-vanishing where $i:Z\hookrightarrow M$ denotes the inclusion.}

Let us denote by $\alpha_i=\pi^*(\widehat{\alpha_i})$ with $\widehat{\alpha_i}$ one-forms on $Z$. Since
$${\omega_\epsilon}= df_{\epsilon}\wedge (\sum_{i=0}^{2k} x^i\alpha_i)+\beta=\epsilon^{-(2k+1)} {f^{\prime}}\!\left(\frac{x}{\epsilon}\right) dx \wedge
(\sum_{i=0}^{2k} x^i \alpha_i)+\beta,$$
{
which on the interval
 $\vert x\vert <\epsilon$ is equal to}
 {$\epsilon^{-(2k+1)}( f^{\prime}\!\left(\frac{x}{\epsilon}\right) dx \wedge \alpha_0  +\mathcal{O}(\epsilon))+ \beta$}
\noindent and hence  {$${\omega_\epsilon}^n=\epsilon^{-(2k+1)}(f^{\prime}\!\left(\frac{x}{\epsilon}\right) dx \wedge
\alpha_0\wedge\beta^{n-1}+\mathcal{O}(\epsilon))$$}{ Now observe that  for $\epsilon$ sufficiently small ${\omega_{\epsilon}}^n$ is transverse to the zero section of  $\Lambda^n(T^*M)$ because on one hand $dx\wedge\alpha_0\wedge \beta^{n-1}\neq 0$  (by Proposition \ref{prop:thetamu}) and on the other,  by definition of $f$, $f(x)=-x^2+2$ so $f^{\prime}(x)=-2x$ vanishes linearly at $x=0$.} { Observe too that $(i^*{\omega_{\epsilon}})^{n-1}$ is non-vanishing because $\beta$ induces a symplectic structure on each of the symplectic leaves of the original $b^{2k+1}$-symplectic structure $\omega$, and therefore $\beta^{n-1}\neq 0$. Thus ${\omega_\epsilon}$ is a folded symplectic structure.}

%
%
%
%
%

\end{proof}
As a consequence of this fact we obtain the following theorem which generalizes  some of the results contained in Section 3 in
\cite{frejlichmartinezmiranda} for $b$-symplectic manifolds\footnote{{ In \cite{frejlichmartinezmiranda} it was already noticed that $b$-symplectic manifolds share many properties with folded symplectic manifolds in particular in Theorem A in \cite{frejlichmartinezmiranda} it was proved that an orientable open manifold $M$ is $b$-symplectic
if and only if $M\times \mathbb C$ is almost-complex. On the other hand observe that the existence of an almost-complex structure on $M\times \mathbb C$ is equivalent to the existence of a folded symplectic form as proved in \cite{cannas}. }}:

\begin{theorem}
A manifold admitting a  $b^{2k+1}$-symplectic structure also admits a folded symplectic structure.
\end{theorem}

\section{ Group actions and desingularization}

  We conclude by briefly mentioning some applications of desingularization which we propose  to explore in detail in a sequel to this paper.

In the papers \cite{gmps} and \cite{gmps2} it was shown that two classical theorems in equivariant symplectic geometry, the Delzant theorem and
the Atiyah-Guillemin-Sternberg convexity theorem, have analogs for $b$-symplectic manifolds. We will show that  these theorems also have analogs
for for $b^m$-manifolds (except for the assertion in Delzant's theorem that \lq\lq the moment image of $M$ determines $M$ up to
symplectomorphism").

In addition we will use the desingularization procedure to prove  $b^m$-versions of the Kirwan convexity theorem  and of the Duistermaat-Heckman
theorem (concerning the latter the main ingredient in our proof will be the observation that in the vicinity of the critical hypersurface $Z$,
the desingularized Duistermaat-Heckman measure can be easily computed and its behavior as $\epsilon$ tends to zero easily described using
Theorems \ref{thm:desingularizingsymp} and \ref{thm:desingularizingfolded}.)

Finally we note that the complexities that are required  to keep track of the \lq\lq infinities" occurring in the $b^m$-versions of the theorems
above can largely be avoided by viewing these infinities as coming from the desingularization process as $\epsilon$ tends to zero.

\subsection{ A convexity result for $b^m$-symplectic manifolds} In what follows we will assume for simplicity that the hypersurface  $Z$ is
connected; this assumption can readily be removed.
{As done in} \cite{gmps}  for $b$-symplectic manifolds, we can define Hamiltonian actions in the $b^m$-setting.

\begin{definition} An action of a torus $G= \mathbb{T}^n$ on the $b^m$-symplectic manifold $(M,\omega)$ is called \textbf{Hamiltonian} if it
preserves $\omega$ and $\iota_{X^\#}\omega$ is $b^m$-exact for  any $X \in\mathfrak{g}$.
\end{definition}

Given such a Hamiltonian  action on a $b^m$-manifold $M$,  this action is also Hamiltonian with respect to the desingularized forms. Hence if
$m$ is even, we obtain a family of symplectic forms and a family of Hamiltonian actions on the pairs $(M, \omega_{\epsilon})$. Observe in this
case the desingularized forms are symplectic and
we can invoke the Atiyah-Guillemin-Sternberg convexity theorem for the moment map (\cite{atiyah}, \cite{guilleminsternberg}).

Let us denote by $F_{\epsilon}$ the associated family of moment maps. Then the images of these moment maps are convex polytopes.  To describe
those polytopes, there are two cases to consider\footnote{By analogs of the results in \cite{gmps},  these are the only two cases that occur,
even when the number of connected components of $Z$ is greater than $1$.}:

\paragraph{Case 1} \emph{ The image of the moment map coincides with the image of the moment map induced in the symplectic foliation in the
critical set }\footnote{In this case all the connected components for the initial action have zero modular weight. Cf. Section 6 in \cite{gmps}.}. In this
case the moment polytope for $M$  coincides with the image of the moment map on one of the symplectic leaves in $Z$.

 \paragraph{Case 2} \emph{ The function $f_\epsilon$ is one of the components of the moment map}\footnote{ In this case, all the connected
components for the initial action have non-zero modular weight. Cf. Section 6 in  \cite{gmps}.}. These polytopes are as depicted in the picture on the left below {where we use the explicit expression of $f_{\epsilon}$}:

\definecolor{qqqqff}{rgb}{0.,0.,1.}
\definecolor{zzttqq}{rgb}{0.6,0.2,0.}
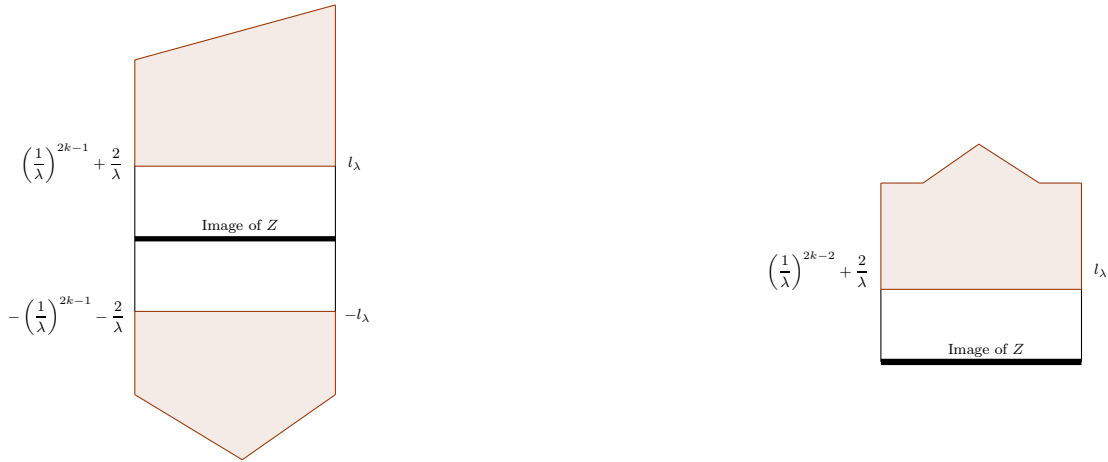
\begin{figure}[!h]\centering\begin{tikzpicture}[scale=0.7, transform shape][line cap=round,line join=round,>=triangle 45,x=1.0cm,y=1.0cm]\clip(-4.3,-5.) rectangle (10.02,6.3);
\fill[color=zzttqq,fill=zzttqq,fill opacity=0.10000000149011612] (-1.,4.26) -- (-1.,2.24) -- (2.809649122807016,2.24) -- (2.809649122807016,5.305438596491227) -- cycle;\fill[color=zzttqq,fill=zzttqq,fill opacity=0.10000000149011612] (-1.,-0.52) -- (-1.,-2.101228070175437) -- (1.04,-3.34) -- (2.809649122807016,-2.1012280701754373) -- (2.809649122807016,-0.52) -- cycle;
\draw [color=zzttqq] (-1.,4.26)-- (-1.,2.24);
\draw [color=zzttqq] (-1.,2.24)-- (2.809649122807016,2.24);
\draw [color=zzttqq] (2.809649122807016,2.24)-- (2.809649122807016,5.305438596491227);
\draw [color=zzttqq] (2.809649122807016,5.305438596491227)-- (-1.,4.26);
\draw [color=zzttqq] (-1.,-0.52)-- (-1.,-2.101228070175437);
\draw [color=zzttqq] (-1.,-2.101228070175437)-- (1.04,-3.34);
\draw [color=zzttqq] (1.04,-3.34)-- (2.809649122807016,-2.1012280701754373);
\draw [color=zzttqq] (2.809649122807016,-2.1012280701754373)-- (2.809649122807016,-0.52);
\draw [color=zzttqq] (2.809649122807016,-0.52)-- (-1.,-0.52);
\draw [line width=2.0pt] (-1.,0.86)-- (2.809649122807016,0.86);
\draw (-1.,2.24)-- (-1.,-0.52);
\draw (2.809649122807016,2.24)-- (2.809649122807016,-0.52);\begin{scriptsize}\draw[color=black] (-2.2,2.3) node {$\displaystyle \left(\frac{1}{\lambda}\right)^{2k - 1} + \frac{2}{\lambda}$};
\draw[color=black] (3.18,2.3) node {$l_\lambda$};
\draw[color=black] (-2.3,-0.6) node {$-\displaystyle \left(\frac{1}{\lambda}\right)^{2k - 1} - \frac{2}{\lambda}$};
\draw[color=black] (3.23,-0.6) node {$-l_\lambda$};
\draw[color=black] (1,1.09) node {Image of $Z$};
\end{scriptsize}\end{tikzpicture}\begin{tikzpicture}[scale=0.7, transform shape][line cap=round,line join=round,>=triangle 45,x=1.0cm,y=1.0cm]\clip(-4.151110350871816,-2.657924867316483) rectangle (8.756349599086468,7.5274310987818565);
\fill[color=zzttqq,fill=zzttqq,fill opacity=0.10000000149011612] (-1.,2.24) -- (-1.,4.26) -- (-0.2,4.26) -- (0.8604453280226013,5) -- (2.009649122807016,4.26) -- (2.809649122807016,4.26) -- (2.809649122807016,2.24) -- cycle;
\draw [line width=2.4pt] (-1.,0.86)-- (2.809649122807016,0.86);
\draw (-1.,0.86)-- (-1.,2.24);
\draw (2.8096491228070164,0.86)-- (2.809649122807016,2.24);
\draw [color=zzttqq] (-1.,2.24)-- (-1.,4.26);
\draw [color=zzttqq] (-1.,4.26)-- (-0.2,4.26);
\draw [color=zzttqq] (-0.2,4.26)-- (0.8604453280226013,5);
\draw [color=zzttqq] (0.8604453280226013,5)-- (2.009649122807016,4.26);
\draw [color=zzttqq] (2.009649122807016,4.26)-- (2.809649122807016,4.26);
\draw [color=zzttqq] (2.809649122807016,4.26)-- (2.809649122807016,2.24);
\draw [color=zzttqq] (2.809649122807016,2.24)-- (-1.,2.24);
\begin{scriptsize}\draw[color=black] (-2.2,2.6) node {$\displaystyle \left(\frac{1}{\lambda}\right)^{2k -2} + \frac{2}{\lambda}$};
\draw[color=black] (3.18,2.6) node {$l_\lambda$};
\draw[color=black] (1,1.09) node {Image of $Z$};
\end{scriptsize}\end{tikzpicture}\caption{ On the left moment map image of desingularization, even case. In the figure the region above $l_{\lambda}$  and the region below $-l_\lambda$ are independent of $\epsilon$ for $\epsilon <\lambda$. On the right moment map image of desingularization, odd case.}\end{figure}
Finally for $b^{2k+1}$-symplectic manifolds the desingularization gives us folded symplectic manifolds and for these the moment polytopes are
the folded versions of the polytopes above,  as depicted in the figure above on the right.

  We will give a more detailed and rigorous account of these results in a future paper.
 Similarly we will prove analogues of the Duistermaat-Heckman theorems and the Delzant theorem using these methods.

\end{document}